\documentclass[11pt,twoside,reqno,psamsfonts]{amsart}

\usepackage[OT1]{fontenc}
\usepackage{type1cm}
\usepackage{amssymb}
\usepackage[left=2.3cm,top=3cm,right=2.3cm]{geometry}

\numberwithin{equation}{section}

\theoremstyle{plain}
\newtheorem{theorem}{Theorem}[section]

\newtheorem{corollary}[theorem]{Corollary}

\newtheorem{proposition}[theorem]{Proposition}

\newtheorem{lemma}[theorem]{Lemma}

\theoremstyle{remark}
\newtheorem{remark}[theorem]{Remark}

\newtheorem{ex}[theorem]{Example}

\theoremstyle{definition}

\newcommand{\BB}{\mathcal{B}}

\newcommand{\HH}{\mathcal{H}}

\newcommand{\MM}{\mathcal{M}}

\newcommand{\QQ}{\mathcal{Q}}
\newcommand{\WW}{\mathcal{W}}

\newcommand{\R}{\mathbb{R}}

\newcommand{\Z}{\mathbb{Z}}
\newcommand{\N}{\mathbb{N}}

\newcommand{\roo}{\varrho}

\DeclareMathOperator{\dist}{dist}
\DeclareMathOperator{\diam}{diam}
\DeclareMathOperator{\proj}{proj}

\DeclareMathOperator{\length}{length}

\DeclareMathOperator{\dima}{dim_A}
\DeclareMathOperator{\udima}{\overline{dim}_A}
\DeclareMathOperator{\ldima}{\underline{dim}_A}

\DeclareMathOperator{\ucodima}{\overline{co\,dim}^{\,\mu}_A}
\DeclareMathOperator{\lcodima}{\underline{co\,dim}^{\,\mu}_A}

\DeclareMathOperator{\dimreg}{dim_{reg}}
\DeclareMathOperator{\udimreg}{\overline{dim}_{reg}}
\DeclareMathOperator{\ldimreg}{\underline{dim}_{reg}}

\DeclareMathOperator{\dimh}{dim_H}

\DeclareMathOperator{\dimm}{dim_M}
\DeclareMathOperator{\udimm}{\overline{dim}_M}
\DeclareMathOperator{\ldimm}{\underline{dim}_M}

\DeclareMathOperator{\ucodimm}{\overline{co\,dim}^{\,\mu}_M}
\DeclareMathOperator{\lcodimm}{\underline{co\,dim}^{\,\mu}_M}

\DeclareMathOperator{\ldims}{\underline{dim}_S}

\newcommand{\bdry}{\partial}

\newcommand{\Ha}{\mathcal H}

\newcommand{\Mi}{\mathcal M}

\newcommand{\sub}{\subset}

\newcommand{\Char}[1]{\chi_{\lower 1.5pt\hbox{$\scriptscriptstyle #1$}}}

\newcommand{\capp}[1]{\operatorname{cap}_{#1}}

\providecommand{\ch}[1]{\text{\raise 2pt \hbox{$\chi$}\kern-0.2pt}_{#1}}

\begin{document}

\title[Dimensions, Whitney covers, and tubular neighborhoods]{Dimensions, Whitney covers, and tubular neighborhoods}

\author{Antti K\"aenm\"aki}
\author{Juha Lehrb\"ack}
\address{Department of Mathematics and Statistics \\
         P.O. Box 35 (MaD) \\
         FI-40014 University of Jyv\"askyl\"a \\
         Finland}
\email{antti.kaenmaki@jyu.fi}
\email{juha.lehrback@jyu.fi}

\author{Matti Vuorinen}
\address{Department of Mathematics and Statistics \\
         University of Turku \\
         Assistentinkatu 7 \\
         FI-20014 Turku \\
         Finland}
\email{vuorinen@utu.fi}

\thanks{JL acknowledges the support of the Academy of Finland (project \# 252108), MV acknowledges the support of the Academy of Finland (project \#  2600066611) }
\date{\today}

\begin{abstract}
  Working in doubling metric spaces, we examine the connections between different dimensions, Whitney covers, and geometrical properties of tubular neighborhoods. In the Euclidean space, we relate these concepts to the behavior of the surface area of the boundaries of parallel sets. In particular, we give characterizations for the Minkowski and the spherical dimensions by means of the Whitney ball count.
\end{abstract}

\maketitle

\section{Introduction}

Various notions of dimension reflect the structure of sets in different ways. In this article, the main interest is directed towards the Minkowski and Assouad dimensions. Common to both of these is that they are defined using covers consisting of balls with a fixed radius, contrary to the Hausdorff dimension, where more flexibility for covers is allowed.
The main difference between Minkowski and Assouad dimensions is, roughly speaking, that the former is related to the average small scale structure of sets, while the latter depends on the extreme properties of sets and takes into account all scales.
The goal of this work is to understand the relationships between these dimensions and the behavior of (tubular or annular or spherical) neighborhoods of a given set. In this respect, also the role of geometrical conditions, such as porosity and uniform perfectness, is examined. Our study continues and combines e.g.\ the works of \cite{Aikawa1991, AikawaEssen1996, Assouad1983, Bouligand1928, JarviVuorinen1996, LehrbackTuominen2012, Luukkainen1998, MartioVuorinen1987, RatajWinter2010, Vuorinen1987}.
A common theme to most of our results is that we explore means to obtain knowledge on the internal structure of a closed set $E\sub X$ when only external information is available.
Such external information can be given, for instance, in terms of a Whitney type cover of the complement $X\setminus E$.  One motivation for this kind of study originates from the geometric function theory: How are good properties of domains related to the behavior of their boundaries?

In the first part of the article, Sections \ref{sect:metric} and \ref{sect:metric measure}, we recall definitions and preliminary results, but also record some new observations concerning dimensions and geometric conditions. More precisely,
in Section \ref{sect:metric} we review the setting of a doubling metric space and the notions of different dimensions, and recall some of the interrelations and other basic properties of these dimensions. Whitney covers, generalizations of the classical Whitney decompositions (see \cite{Stein1970}) to more general metric spaces, are reviewed at the end of Section \ref{sect:metric}; recall that Whitney covers are a standard tool in analysis, used for instance in various extension theorems and in the study of singular integrals.

Most of Section \ref{sect:metric measure} deals with the interplay between dimensions, codimensions, and doubling measures. As a particular by-product of these considerations, we improve the classical dimension estimate of porous sets: If the space is $s$-regular, then the upper Assouad dimension of $E$ is at most $s-c\roo^s$ for all $\roo$-porous $E$; cf.\ \cite[Theorem 4.7]{JarvenpaaJarvenpaaKaenmakiRajalaRogovinSuomala2007}.

In Section \ref{sect:whitney}, we perform a systematic study on the relations between the Whitney ball count, Minkowski dimensions, and Assouad dimensions. In the Euclidean case, the connection between the upper Minkowski dimension and upper bounds for Whitney cube count was considered in \cite{MartioVuorinen1987}. Besides extending these results and their local counterparts to more general spaces, we also establish corresponding results for lower bounds and lower dimensions. We remark here that the porosity of the set $E$ plays an important role when lower bounds for Whitney ball count are considered. In particular, we characterize the Minkowski dimensions of a compact porous set $E$ by the limiting behavior of the Whitney ball count of the complement of $E$.

The final Section \ref{sect:tube} is devoted to another type of external information, special to the Euclidean space $\R^d$. Namely, here we study the behavior of the boundaries of the parallel sets $E_r$, the so-called $r$-boundaries, of closed sets $E\sub\R^d$. For earlier results concerning these sets, see for instance \cite{Brown1972, Ferry1975, Fu1985, GariepyPepe1972, Jarvi1994, OleksivPesin1985, RatajWinter2010, Stacho1976, Winter2011}. Continuing this line of research, we show that there is in fact an intimate connection between the `surface area' $\HH^{d-1}(\partial E_r)$ and the Whitney ball count of $\R^d\setminus E$. One particular outcome of this connection is the result that if $E \subset \R^d$ is compact and $s$-regular for $0<s<d$, then $cr^{d-1-s} \le \HH^{d-1}(\partial E_r) \le Cr^{d-1-s}$ for all sufficiently small $r>0$. As another application we examine the properties of the spherical dimensions of $E$, defined by the limiting behavior of $\HH^{d-1}(\partial E_r)$. We answer an open
question of Rataj and Winter \cite{RatajWinter2010} by constructing a compact set having zero measure and lower Minkowski dimension $d$ but lower spherical dimension $d-1$. Our construction also shows, answering a question of Winter \cite{Winter2011}, that the existence of the Minkowski dimension does not guarantee the equivalence of the lower Minkowski dimension and the lower spherical dimension. Moreover, unlike in the case of the Minkowski dimension, the limiting behavior of the Whitney ball count of the complement of $E$ characterizes the spherical dimensions for all compact sets $E\sub\R^d$.

\section{Dimensions in metric spaces}\label{sect:metric}

Unless otherwise explicitly stated, we work on a \emph{doubling metric space} $(X,d)$: there is $N = N(X) \in \N$ so that any closed ball $B(x,r)$ of center $x$ and radius $r>0$ can be covered by $N$ balls of radius $r/2$. Notice that even if $x \ne y$ or $r\neq t$, it may happen that $B(x,r)=B(y,t)$. For notational convenience, we keep to the convention that each ball comes with a fixed center and radius. This makes it possible to use notation such as $2B=B(x,2r)$ without referring to the center or radius of the ball $B=B(x,r)$. For convenience, we also make the general assumption that the space $X$ contains at least two points.

A doubling metric space is clearly separable. Hence for each $r > 0$ and $ E \subset X$ there exists a maximal $r$-packing of $E$. Recall that a countable collection $\BB$ of pairwise disjoint balls centered at $E$ with radius $r>0$ is called an \emph{$r$-packing} of $E$. It is \emph{maximal} if for each $x \in E$ there is $B \in \BB$ so that $B(x,r) \cap B \ne \emptyset$. We frequently use the fact that if $\{ B_i \}_i$ is a maximal packing of $E$, then $\{ 2B_i \}_i$ is a cover of $E$.

Applying the doubling condition inductively, it follows that there is a constant $C \ge 1$ such that each ball $B(x,R)$ can be covered by at most $C(r/R)^{-s}$ balls of radius $r$ for all $0<r<R<\diam(X)$, where $s = \log_2 N$. Of course, this can also be true for smaller values of $s$. The infimum of such admissible exponents $s$ is called the \emph{upper Assouad dimension} of $X$. Therefore, doubling metric spaces are precisely the metric spaces with finite upper Assouad dimension. Considering the restriction metric, the definition extends to all subsets of $X$. The upper Assouad dimension of $E \subset X$ is denoted by $\udima(E)$.

The above intrinsically metric definition of dimension is essentially due to Assouad (see e.g.\ \cite{Assouad1983}), but, according to \cite[Remark 3.6]{Luukkainen1998}, the origins of a related concept date back to Bouligand \cite{Bouligand1928} whose definition (for $E\sub\R^d$) was external and used also the Lebesgue measure near the set $E$. We refer to Luukkainen \cite{Luukkainen1998} for the basic properties and a historical account on the upper Assouad dimension. We also remark that in the literature, the upper Assouad dimension is more commonly referred to as Assouad dimension.

Conversely to the above definition, we may also consider all $t \ge 0$ for which there is a constant $c>0$ so that if $0<r<R<\diam(X)$, then for every $x \in X$ at least $c(r/R)^{-t}$ balls of radius $r$ are needed to cover $B(x,R)$. We call the supremum of all such $t$ the \emph{lower Assouad dimension} of $X$. Again, the restriction metric is used to define the lower Assouad dimension $\ldima(E)$ of a subset $E\sub X$. If $\ldima(E)=\udima(E)=s$, then we say that $E$ is \emph{uniformly $s$-homogeneous} and we denote $s$ by $\dima(E)$.

A metric space $X$ is \emph{uniformly perfect} if there exists a constant $C \ge 1$ so that for every $x \in X$ and $r>0$ we have $B(x,r) \setminus B(x,r/C) \ne \emptyset$ whenever $X \setminus B(x,r) \ne \emptyset$.

\begin{lemma}\label{lemma:unif.perf}
  A metric space $X$ is uniformly perfect if and only if $\ldima(X) > 0$.
\end{lemma}

\begin{proof}
  If $X$ contains at least two points and is uniformly perfect with a constant $C \ge 1$, then it follows from \cite[Remark 3.6(iii)]{RajalaVilppolainen2009} that $\ldima(X) \ge \log 2/\log(2C+1) > 0$.

  Suppose that $\ldima(X) > 0$, but $X$ is not uniformly perfect. If $0 < s < \ldima(X)$ and $B$ is any ball of radius $0<R<\diam(X)$, then $\#\BB \ge c(r/R)^{-s}$ for all maximal $r$-packings $\BB$ of $B$. Choosing now $C$ so that $C^s \ge 2/c$, we find $x \in X$ and $R>0$ so that $B(x,R) \setminus B(x,R/C) = \emptyset$. This gives a contradiction since $B(x,R)$ can clearly be covered by one ball of radius $R/C$, but the previous estimate says that each maximal $R/C$-packing of $B(x,R)$ contains at least two balls.
\end{proof}

We will frequently use the Hausdorff dimension $\dimh$ and the $\lambda$-dimensional Hausdorff measure $\HH^\lambda$ and its $r$-content $\HH^\lambda_r$ (cf.\ \cite[\S 4]{Mattila1995}).
Recall also that the \emph{$\lambda$-dimensional Minkowski $r$-content} of a compact set $E \subset X$ is
\begin{equation*}
  \MM^\lambda_r(E) = \inf\{ nr^\lambda : E \subset \bigcup_{k=1}^n B(x_k,r),\; x_k \in E \}
\end{equation*}
and the \emph{upper and lower Minkowski dimensions} are defined, respectively, as
\begin{align*}
  \udimm(E) &= \inf\big\{ \lambda \ge 0 : \limsup_{r \downarrow 0} \MM^\lambda_r(E) = 0 \big\}, \\
  \ldimm(E) &= \inf\big\{ \lambda \ge 0 : \liminf_{r \downarrow 0} \MM^\lambda_r(E) = 0 \big\}.
\end{align*}
In the case when $\udimm(E) = \ldimm(E)$, the common value is denoted by $\dimm(E)$. If the Assouad dimension is determined by looking at the number of balls needed to cover the set locally everywhere, then the Minkowski dimension tells how many balls are needed in average. Recall that if $E \subset X$ is compact, then $\dimh(E) \le \ldimm(E) \le \udimm(E) \le \udima(E)$.

\begin{lemma}\label{lemma:ala-ass & h-dorff}
  If $X$ is complete and $E \subset X$ is closed, then $\ldima(E) \le \dimh(E \cap B_0)$ 
  for all balls $B_0$ centered at $E$.
\end{lemma}

\begin{proof}
  Notice first that if $0<t_0<\ldima(E)$, then for each ball $B \subset X$ centered at $E$ with radius $0<R<\diam(E)$ we have, by the definition of the lower Assouad dimension, that $\MM^{t_0}_r(E \cap B) \geq c_0 R^{t_0}$ for all $0<r<R$, where $c_0>0$ does not depend on $r$ nor the choice of the ball $B$. From such a uniform estimate we obtain, by \cite[Lemma 4.1]{Lehrback2009}, that for each $0<t<t_0$ there is a constant $c>0$ such that
  \begin{equation}\label{eq:unif.fat}
    \HH^{t}_R\big(E \cap B(x,R)\big)\geq c R^t\quad\text{ whenever $x\in E$ and $0<R<\diam(E)$}.
  \end{equation}
  The proof in \cite{Lehrback2009} is written for closed subsets of a Euclidean space, but it applies verbatim also in complete doubling metric spaces; the idea is to use the uniform Minkowski content estimate repeatedly to construct a Cantor-type subset $C\subset E\cap B_0$ for which estimate \eqref{eq:unif.fat} holds.
Therefore $\dimh(E\cap B_0)\geq\dimh(C)\geq t$, and the claim follows.
\end{proof}

\begin{remark}\label{rmk:thick}
  The proof of Lemma \ref{lemma:ala-ass & h-dorff} shows that condition \eqref{eq:unif.fat} holds for each $0\leq t<\ldima(E)$. Conversely, if \eqref{eq:unif.fat} holds, then it is clear that also
$\MM^{t}_r(E \cap B(x,R))\geq c R^t$ for all $0<r<R$, and thus $\ldima(E)\geq t$. Hence,
 in a complete space,
\[
  \ldima(E) = \sup\{ t\geq 0 : \eqref{eq:unif.fat} \text{ holds} \}.
\]
\end{remark}

\begin{remark}\label{rmk:unif perf}
The notion of uniform perfectness of subsets of the complex
plane was perhaps first studied in \cite{Pommerenke1979}. About the same time, the
equivalent notion of homogeneously dense set, in the setting of general
metric spaces, was introduced in \cite{TukiaVaisala1980}. In \cite{Vuorinen1987}, a metric thickness
condition was defined in terms of capacity densities and in \cite[Theorem 4.1]{JarviVuorinen1996} it was
shown that for a closed set $E\sub \R^d$ this thickness condition is equivalent to uniform perfectness
and also to the existence of $t>0$ such that \eqref{eq:unif.fat} holds; compare this to Lemma~\ref{lemma:unif.perf} and Remark~\ref{rmk:thick}.
\end{remark}

Let $\Omega \subset X$ be an open set such that $\Omega \ne X$. By a Whitney cover we mean a collection of closed balls having radii comparable to the distance to the closed set $X \setminus \Omega$ and with controlled overlap. More precisely, a \emph{Whitney cover} $\WW(\Omega)$ is a countable collection $\{ B_i \}_{i\in I}$ of balls $B_i = B(x_i,r_i)$ with $x_i \in \Omega$ and $r_i = \tfrac18\dist(x_i,X \setminus \Omega)$ such that $\Omega=\bigcup_{i\in I} B_i$ and
$\sum_{i\in I} \Char{B_i}\leq C$ for some constant $C \ge 1$. Here $\Char{E}$ is the characteristic function of $E\sub X$.
Such a collection can be constructed from maximal packings of the sets $(X \setminus \Omega)_{2^{k}} \setminus (X \setminus \Omega)_{2^{k-1}}$, $k\in\Z$;
see e.g.\ \cite[Proposition 4.1.15]{HeinonenKoskelaShanmugalingamTyson2011} for details.
Here $E_r = \{ x \in X : \dist(x,E) < r \}$ is the open \emph{$r$-neighborhood} of $E$. It is obvious that for a given open set $\Omega\sub X$ there usually exist many Whitney covers. Note also that the choice of the constant $\frac 1 8$ is not important. If we used radii $r_i = \delta\dist(x_i,X \setminus \Omega)$ for $0<\delta\leq \frac 12$ instead, then our
results would remain valid up to constants depending on $\delta$ and, in some particular cases, other modifications which
will be commented later; see, for instance~\ref{rmk:small c}.

For $k \in \Z$ and $A \subset X$ we set
\[
  \WW_k(\Omega;A) = \{ B(x_i,r_i) \in \WW(\Omega) : 2^{-k-1} < r_i \le 2^{-k} \text{ and } A \cap B(x_i,r_i) \ne \emptyset \}
\]
and $\WW_k(\Omega) = \WW_k(\Omega;X)$.
Observe that if $B_1, B_2 \in \WW(\Omega)$ and $B_1\cap B_2\neq \emptyset$, then
there is $k\in \Z$ such that $B_1, B_2 \in \WW_k(\Omega)\cup\WW_{k+1}(\Omega)$.

Recall that the \emph{Whitney decomposition} of an open set $\Omega \subset \R^d$ is $\Omega = \bigcup_{Q \in \WW} Q$, where the cubes $Q \in \WW$ have pairwise disjoint interiors, edges parallel to coordinate axes, and their diameters are of the form $\diam(Q) = 2^{-k}$, $k\in\Z$. Moreover, the diameters satisfy the condition
\[
  \diam(Q) \le \dist(Q,\partial\Omega) \le 4\diam(Q).
\]
See \cite{Stein1970} for the existence and properties of Whitney decompositions. It is clear that a Whitney decomposition introduces a Whitney cover having exactly the same cardinality (but with constant $\delta = \tfrac12$).

\section{Dimensions in metric measure spaces}\label{sect:metric measure}

By a \emph{measure}, we exclusively refer to a nontrivial Borel regular outer measure for which bounded sets have finite measure.
We say that a measure $\mu$ on $X$ is \emph{doubling} if there is a constant $C \ge 1$ so that
\begin{equation} \label{eq:tuplamitta}
  0 < \mu(2B) \le C\mu(B)
\end{equation}
for all closed balls $B$ of $X$.

The existence of a doubling measure yields an upper bound for the upper Assouad dimension of $X$. Indeed, let $C \ge 1$ be as in \eqref{eq:tuplamitta}, $s = \log_2 C$, and apply \eqref{eq:tuplamitta} inductively to find a constant $c>0$ for which
\begin{equation} \label{eq:doubling dimension}
  \frac{\mu(B(y,r))}{\mu(B(x,R))}\ge c\Bigl(\frac rR\Bigr)^s
\end{equation}
for all $y\in B(x,R)$ and $0<r<R<\diam(X)$.
Again, the estimate \eqref{eq:doubling dimension} may also hold for smaller values of $s$. The infimum of such admissible exponents $s$ is called the \emph{upper regularity dimension} of $\mu$, denoted by $\udimreg(\mu)$. A simple volume argument implies that $\udima(X) \le \udimreg(\mu)$ whenever $\mu$ is a doubling measure on $X$. On the other hand, Vol'berg and Konyagin \cite[Theorem 4]{VolbergKonyagin1987} have given an example of a space $X$ where this inequality is strict for all doubling measures $\mu$. In particular, if a metric space carries a doubling measure, then the space is doubling. The reverse implication is true if $X$ is assumed to be complete; see \cite{VolbergKonyagin1984, VolbergKonyagin1987, LuukkainenSaksman1998, Wu1998, BylundGudayol2000, KaenmakiRajalaSuomala2012a}.

We have the following converse to \eqref{eq:doubling dimension}.

\begin{lemma} \label{thm:converse dbl}
  If $X$ is uniformly perfect and $\mu$ is doubling, then there exist $t>0$ and $C \ge 1$ such that
  \begin{equation}\label{eq:converse dbl}
    \frac{\mu(B(y,r))}{\mu(B(x,R))}\leq C\Bigl(\frac rR\Bigr)^t
  \end{equation}
  whenever $0<r<R<\diam(X)$ and $y\in B(x,R)$.
\end{lemma}

\begin{proof}
  It follows from \cite[Proposition B.4.6]{Gromov1999} that there exists $0<\delta<1$ so that $\mu(B(y,\delta^k R)) \le (1-\delta)^k \mu(B(y,R))$ for all $y \in X$, $0<R<\diam(X)$, and $k \in \N$. If $0<r<R$ and $k \in \N$ is such that $\delta^{k+1}R < r \le \delta^k R$, then
  \begin{equation*}
    \frac{\mu(B(y,r))}{\mu(B(y,R))} \le (1-\delta)^k \le (1-\delta)^{-1} \Bigl( \frac{r}{R} \Bigr)^{\log(1-\delta)/\log\delta}.
  \end{equation*}
  The proof is finished since $\mu(B(y,R)) \le \mu(B(x,2R)) \le C\mu(B(x,R))$.
\end{proof}

The supremum of all admissible exponents $t$ in \eqref{eq:converse dbl} is called the \emph{lower regularity dimension} of $\mu$, denoted by $\ldimreg(\mu)$. In particular, Lemma \ref{thm:converse dbl} says that $\ldimreg(\mu)>0$ for a doubling measure $\mu$ in a uniformly perfect space. It follows directly from the estimate \eqref{eq:converse dbl} that $\ldimreg(\mu) \le \ldima(X)$. If $X$ is not uniformly perfect, then, by recalling Lemma \ref{lemma:unif.perf}, it is natural to define $\ldimreg(\mu)=0$. In the case when $\udimreg(\mu) = \ldimreg(\mu)$, the common value is denoted by $\dimreg(\mu)$.

The measure $\mu$ is \emph{$s$-regular} (for $s > 0$) if there is a constant $C \ge 1$ such that
\[
  C^{-1}r^s \le \mu(B(x,r)) \le Cr^s
\]
for all $x \in X$ and every $0 < r < \diam(X)$. It is immediate that if $X$ is bounded, then a measure $\mu$ is $s$-regular if and only if $\dimreg(\mu) = s$
(and `only if' part holds for any $X$). Thus, if $X$ is \emph{$s$-regular}, i.e.\ it carries an $s$-regular measure $\mu$, then $X$ is uniformly $s$-homogeneous. A subset $E \subset X$ is \emph{$s$-regular} if it is an $s$-regular space in the relative metric.

\begin{remark}
If $X$ is $s$-regular and a closed subset $E\subset X$ is such that \eqref{eq:unif.fat} holds for all $0<R<\diam(X)$ with $t<s-1$
(in particular, if $t<\min\{\ldima(E),s-1\}$), then $E$ satisfies a uniform capacity density condition:
$E$ is \emph{uniformly $p$-fat} for all $p>s-t$, that is,
\[
  \capp{p}(E\cap B,2B)\geq C \capp{p}(B,2B)
\]
for all balls $B\sub X$ of radius $0<r<\diam(X)$. Here $\capp{p}(K,\Omega)$ is the
variational $p$-capacity of a compact set $K$ with respect to
the open set $\Omega\supset K$. See e.g.~\cite{Costea2009} for the
definitions; the above result can be deduced from~\cite[Remark~4.5]{Costea2009}.
Recall that in an $s$-regular space, $\capp{p}(B,2B)$ is comparable to $r^{s-p}$ for all balls $B$ of radius $r>0$.

Conversely, if
a closed set $E\subset X$ is uniformly $p$-fat for some $1<p\leq s$,
then estimate \eqref{eq:unif.fat}
holds with $t=s-p$ for all $0<R<\diam(X)$. This follows e.g.\ from~\cite[Theorem 4.9]{Costea2009}.
But then, by Remark~\ref{rmk:thick},
$\ldima(E)\geq t=s-p$.
Combining the above observations, we conclude that in an $s$-regular complete space $X$
the lower Assouad dimension can also be characterized as
\[
\ldima(E)=s-\inf\{1<p\leq s : E \text{ is uniformly $p$-fat}\}
\]
for a subset $E\sub X$ with $\ldima(E) < s-1$; in the case when $X$ is unbounded we need to further assume that $E$ is unbounded as well. The relation between uniform $s$-fatness and uniform perfectness in an $s$-regular metric space (recall Remark \ref{rmk:thick}) has also been considered in \cite{KorteShanmugalingam2010}.
\end{remark}

In doubling metric spaces, it is sometimes more convenient to use modified version of the Minkowski content, namely the \emph{Minkowski $r$-content of codimension $q$}. Given a doubling measure $\mu$, this is defined by setting
\[
  \MM^{\mu,q}_r(E) = \inf\bigl\{ r^{-q} \sum_{k} \mu(B(x_k,r)) : E \subset \bigcup_{k} B(x_k,r),\; x_k \in E \bigr\}
\]
for all compact sets $E \subset X$. Note that $\MM^{\mu,q}_r(E)$ is comparable to $r^{-q} \mu(E_r)$.
\emph{The lower and upper Minkowski codimensions of $E$} are defined, respectively, as
\begin{align*}
  \lcodimm(E) &= \sup\{ q \ge 0 : \limsup_{r \downarrow 0} \MM_r^{\mu,q}(E) = 0 \}, \\
  \ucodimm(E) &= \sup\{ q \ge 0 : \liminf_{r \downarrow 0} \MM_r^{\mu,q}(E) = 0 \}.
\end{align*}
If $s > \udimreg(\mu)$ and $t < \ldimreg(\mu)$, then there is a constant $C \ge 1$ so that $C^{-1}r^{\lambda-t}\mu(E_r) \le \MM^\lambda_r(E) \le Cr^{\lambda-s}\mu(E_r)$ for all $0<r<\diam(E)$ (cf.\ \cite[Remark 2.3]{Lehrback2012} and the proof of \cite[Lemma~3.3]{JarvenpaaJarvenpaaKaenmakiRajalaRogovinSuomala2007}). Thus
\begin{equation}\label{eq:dimm ja codimm}
\begin{split}
  \ldimreg(\mu) &\le \lcodimm(E) + \udimm(E) \le \udimreg(\mu), \\
  \ldimreg(\mu) &\le \ucodimm(E) + \ldimm(E) \le \udimreg(\mu)
\end{split}
\end{equation}
for all compact sets $E \subset X$ and doubling measures $\mu$.

We also consider the following localizations of the Minkowski codimensions,
in the spirit of Bouligand \cite{Bouligand1928}.
If $\mu$ is a doubling measure and $E \subset X$, then we say that the \emph{lower Assouad codimension}, denoted by $\lcodima(E)$, is the supremum of
those $t\ge 0$ for which there exists a constant $C \ge 1$ such that
\begin{equation}\label{eq:bouli}
\frac{\mu(E_r\cap B(x,R))}{\mu(B(x,R))}\le C\Bigl(\frac r R\Bigr)^t
\end{equation}
for every $x\in E$ and all $0<r<R<\diam(E)$.
The \emph{upper Assouad codimension} of $E\sub X$,
denoted by $\ucodima(E)$,
is naturally defined as the infimum of all $s \geq 0$ for which
there is $c>0$ such that
\begin{equation}\label{eq:bouliconv}
\frac{\mu(E_r \cap B(x,R))}{\mu(B(x,R))} \geq c\Bigl(\frac r R\Bigr)^s
\end{equation}
for every $x\in E$ and all $0<r<R<\diam(E)$.

\begin{remark}
It follows from \cite[Theorem 5.1]{LehrbackTuominen2012}, that the
lower Assouad codimension can be characterized as
the supremum of
those $q\geq 0$ for which there exists a constant $C \ge 1$ such that
\begin{equation}\label{eq:aikawa}
  \frac{1}{\mu(B(x,r))} \int_{B(x,r)} \dist(y,E)^{-q}\,d\mu(y) \le C r^{-q}
\end{equation}
for every $x\in E$ and all $0<r<\diam(E)$.
We interpret the integral to be $+\infty$ if $q>0$ and $E$ has positive measure.

A concept of dimension defined via integrals as in \eqref{eq:aikawa}
was first used by Aikawa in \cite{Aikawa1991} for subsets of $\R^d$
(see also \cite{AikawaEssen1996}).
Thus, in \cite{LehrbackTuominen2012}, where the interest originates from such
integral estimates, the lower Assouad codimension is called
the \emph{Aikawa codimension}. However, the
resemblance between the conditions in \eqref{eq:bouli} and \eqref{eq:bouliconv}
and the definitions of upper and lower Assouad dimensions justifies
the present terminology; compare also \eqref{eq:dimm ja codimm} to Lemma \ref{lemma:dima ja codima} below.
\end{remark}

\begin{lemma}\label{lemma:dima ja codima}
If $\mu$ is a doubling measure on $X$ and $E\sub X$, then
\begin{equation}\begin{split}\label{eq:dima ja codima}
  \ldimreg(\mu) &\le \lcodima(E) + \udima(E) \le \udimreg(\mu), \\
  \ldimreg(\mu) &\le \ucodima(E) + \ldima(E) \le \udimreg(\mu).
\end{split}
\end{equation}
\end{lemma}

\begin{proof}
Fix $0<r<R<\diam(E)$.
To prove the left-hand side inequalities of \eqref{eq:dima ja codima}, let $t<\ldimreg(\mu)$
and let $\{B_i\}_{i=1}^n$ be a cover of
$E\cap B(x,2R)$ by balls of radius $r$.
Then the balls $2B_i$ cover $E_r\cap B(x,R)$, and thus
\begin{equation}\label{eq:myy ja N ylos}
\mu(E_r\cap B(x,R))\leq C n \Big(\frac r R\Big)^{t}\mu(B(x,R)).
\end{equation}
Now, if $s > \udima(E)$ we may assume that $n \leq c(r/R)^{-s}$.
It follows that $\lcodima(E)\geq t-s$, and this proves the first left-hand side inequality.
On the other hand, if $s>\ucodima(E)$, then
$\mu(E_r\cap B(x,R))\geq c (r/R)^s \mu(B(x,R))$,
whence \eqref{eq:myy ja N ylos} shows that $n \geq c(r/R)^{s-t}$.
This yields the second left-hand side inequality.

Conversely, if $s>\udimreg(\mu)$
and $\{B_i\}_{i=1}^n$ is an $r$-packing of $E\cap B(x,R/2)$, then
\begin{equation}\label{eq:myy ja N alas}
\mu(E_r\cap B(x,R))\geq c n \Big(\frac r R\Big)^{s}\mu(B(x,R)).
\end{equation}
If $t>\lcodima(E)$, then
we obtain from \eqref{eq:myy ja N alas} that
$n \leq c(r/R)^{t-s}$ and
thus $\udima(E)\leq s-t$ giving the first right-hand side inequality.
Finally, if $t<\ldima(E)$,
then we may assume that $n \geq c(r/R)^{-t}$. Now \eqref{eq:myy ja N alas} implies that
$\ucodima(E)\leq s - t$ yielding the second
right-hand side inequality.
\end{proof}

In particular, if $\mu$ is $s$-regular, then Lemma \ref{lemma:dima ja codima} implies
\begin{equation}\label{eq:dimai ja codima}
\begin{split}
\udima(E) &  = s - \lcodima(E), \\
\ldima(E) &  = s - \ucodima(E)
\end{split}
\end{equation}
for all $E\sub X$. The first equation in \eqref{eq:dimai ja codima}
was also proven in \cite{LehrbackTuominen2012}.
On the other hand, it is not hard to give examples where $\mu$ is doubling
and any given inequality in \eqref{eq:dima ja codima} is strict for a set $E\sub X$;
compare to \cite[Example 4.3]{LehrbackTuominen2012}.

We say that a set $E \subset X$ is \emph{$\roo$-porous} (for $0 \le \roo \le 1$), if for every $x \in E$ and all $0<r<\diam(E)$ there exists a point $y\in X$ such that $B(y,\roo r) \subset B(x,r) \setminus E$. We remark that a more precise name for this porosity condition would be uniform lower porosity. In an $s$-regular complete metric space~$X$ porosity and regularity of sets are closely related concepts: $E \subset X$ is porous if and only if there are $0<t<s$ and a $t$-regular set $F \subset X$ so that $E \subset F$; see \cite[Theorem~5.3]{JarvenpaaJarvenpaaKaenmakiRajalaRogovinSuomala2007} and \cite[Theorem~1.1]{Kaenmaki2007}. In \cite[Theorem~4.7]{JarvenpaaJarvenpaaKaenmakiRajalaRogovinSuomala2007}, it was shown that $\udimm(E) \le s - c\roo^s$ for all $\roo$-porous sets $E$ in an $s$-regular space.
We can now improve this result by using Lemma~\ref{lemma:dima ja codima}:

\begin{proposition}
  If $X$ is $s$-regular, then there is a constant $c>0$ such that $\udima(E) \le s - c\roo^s$ for all $\roo$-porous sets $E \subset X$.
\end{proposition}

\begin{proof}
  The claim follows by first noting that $\lcodima(E) \ge c\roo^s$ by \cite[Corollary~4.6]{JarvenpaaJarvenpaaKaenmakiRajalaRogovinSuomala2007} and then applying \eqref{eq:dimai ja codima}.
\end{proof}

\begin{remark}\label{remark:non-regular porous}
 In doubling metric spaces the relation between porosity and dimension is more subtle; see \cite[Example 4.8 and Theorem 4.9]{JarvenpaaJarvenpaaKaenmakiRajalaRogovinSuomala2007}. Nevertheless,
 if $\mu$ is a doubling measure on $X$, then \cite[Corollary 4.5]{JarvenpaaJarvenpaaKaenmakiRajalaRogovinSuomala2007} implies that each $\roo$-porous set $E \subset X$
 satisfies $\lcodima(E)\ge t$, where $t>0$ only depends on $\roo$ and the doubling constant of $\mu$.
\end{remark}

\section{Whitney covers and dimension}\label{sect:whitney}

We assume throughout this section that $X$ is a doubling metric space.
If $E\sub X$ is a compact set, then it is rather obvious that the number of the Whitney balls in $\WW_k(X\setminus E)$ should be related to the Minkowski dimensions of $E$. In this section we make these relations precise and, in addition, show that for closed sets a similar correspondence holds between `local Whitney ball count' and the Assouad dimensions. In all of these results, we understand that $E$ is a nonempty proper closed subset of $X$ and furthermore, $\WW(X\setminus E)$ refers to an arbitrary (but fixed) Whitney cover of the complement of $E$. As a particular outcome of the considerations in this section we obtain the following characterizations for Minkowski dimensions. Recall that a metric space $X$ is \emph{$q$-quasiconvex} if there exists a constant $q \ge 1$ such that for every $x,y \in X$ there is a curve $\gamma \colon [0,1] \to X$ so that $x = \gamma(0)$, $y = \gamma(1)$, and $\length(\gamma) \le qd(x,y)$.

\begin{theorem}\label{thm:porous char of dimm}
  If $X$ is quasiconvex and $E\sub X$ is compact and porous, then
  \begin{align*}
     \udimm(E) &= \limsup_{k\to\infty}\tfrac{1}{k}\log_2 \#\WW_k(X \setminus E), \\
     \ldimm(E) &= \liminf_{k\to\infty}\tfrac{1}{k}\log_2 \#\WW_k(X \setminus E).
  \end{align*}
\end{theorem}

\begin{proof}
  The characterization for the upper Minkowski dimension follows from Lemma~\ref{lemma:up bound from dim}(2) and Lemma~\ref{lemma:up bound for dim}(2). For the lower dimension use Lemma~\ref{lemma:low bound for dim}(2) and Lemma~\ref{lemma:low bound from dim}(2), instead.
\end{proof}

\begin{theorem}\label{thm:mu=0 char of dimm}
  (1) If $E\sub X$ is compact, then
  \begin{equation*}
     \ldimm(E) \ge \liminf_{k\to\infty}\tfrac{1}{k}\log_2 \#\WW_k(X \setminus E).
  \end{equation*}

  (2) If $\mu$ is an $s$-regular measure on $X$ and $E\sub X$ is a compact set with $\mu(E)=0$, then
  \begin{equation*}
     \udimm(E) = \limsup_{k\to\infty}\tfrac{1}{k}\log_2 \#\WW_k(X \setminus E).
  \end{equation*}
\end{theorem}

\begin{proof}
  The first estimate is a consequence of Lemma~\ref{lemma:low bound for dim}(2) and the characterization for the upper dimension in the second claim follows from Lemma~\ref{lemma:up bound from dim}(2) and Corollary~\ref{coro:regular mu0 bounds}(2).
\end{proof}

The inequality of Theorem \ref{thm:mu=0 char of dimm} may be strict (for non-porous sets), as our Example \ref{ex:thick cantor} shows. As indicated above, the proofs of these two theorems are worked out through several lemmas, in which different upper and lower bounds for Minkowski dimensions are obtained from respective upper and lower bounds for Whitney ball count, and vice versa. In each case we also obtain the corresponding estimates between the local Whitney ball count and Assouad dimensions. Note that separately considered many of these lemmas hold under much weaker assumptions than those in Theorem~\ref{thm:porous char of dimm} or Theorem~\ref{thm:mu=0 char of dimm}.

In the Euclidean case, the relations between the upper Min\-kows\-ki dimension and upper bounds for Whitney ball (cube) count were established by Martio and Vuorinen \cite{MartioVuorinen1987}. Their results can now be viewed as special cases of our more general approach;
in particular, \cite[Theorem 3.12]{MartioVuorinen1987} contains (in $\R^d$) essentially the same information as the characterization of the upper Minkowski dimension in Theorem \ref{thm:mu=0 char of dimm}. On the contrary, to the best of our knowledge our systematic study of the correspondence between the lower Min\-kows\-ki dimension and lower bounds for Whitney ball count, and all the results concerning Assouad dimensions and local Whitney ball count, are new even in the Euclidean case.

We begin with a lemma which gives a general upper bound for the local Whitney ball count.

\begin{lemma}\label{lemma:general bound}
  If $E \subset X$ is a closed set and $0 < \delta \le 1$, then there is a constant $C~\ge~1$, depending only on $X$ and $\delta$, satisfying the following: If $B_0$ is a closed ball of radius $R$ centered at $E$, $0<r<R$, and $\{ B(w_j,r) \}_{j=1}^n$, $w_j \in E$, is a cover of $E \cap 2B_0$, then
  \begin{equation*}
    \# \WW_k(X \setminus E; B_0) \le Cn
  \end{equation*}
  for all $k \in \Z$ with $\delta r \le 2^{-k} \le r$.
\end{lemma}

\begin{proof}
  If $B(x,r') \in \WW_k(X \setminus E; B_0)$ with $\delta r \le 2^{-k} \le r$, then
    $\delta r/2 \le 2^{-k-1} \le r' \le 2^{-k} \le r$.
  Moreover, for each $y \in B(x,r')$ we have
    $\dist(y,E) \le d(y,x) + \dist(x,E) \le r + 8r = 9r$.
  In particular, there is $j \in \{ 1,\ldots,n \}$ so that $y \in B(w_j,10r)$ and thus $B(x,r') \subset B(w_j,10r)$.

  Fix $j \in \{ 1,\ldots,n \}$ and denote $\WW_k\bigl( X \setminus E; B_0 \cap B(w_j,10r) \bigr)$ by $\{ B(x_i,r_i') \}_{i \in I}$. Note that each $B(x_i,r_i')$ is a subset of $B(w_j, 12r)$. Letting $\BB = \{ B(x_i,\delta r/4) \}_{i \in I}$, we have $\sum_{B \in \BB} \Char{2B} \le C_1$, where the constant $C_1 \ge 1$ is as in the definition of the Whitney cover. Let $\BB'$ be a maximal disjoint subcollection of $\BB$ and define $\BB_{B'} = \{ B \in \BB : \text{the center point of $B'$ is in $2B$} \}$ for all $B' \in \BB'$. By the maximality of $\BB'$, we have $\BB = \bigcup_{B' \in \BB'} \BB_{B'}$. Since $\#\BB_{B'} \le C_1$ for all $B' \in \BB'$, we obtain
  \begin{equation*}
    \#\WW_k\bigl( X \setminus E; B_0 \cap B(w_j,10r) \bigr) = \#\BB \le C_1\#\BB' \le C_1C_2\Bigl( \frac{\delta r/4}{12r} \Bigr)^{-s},
  \end{equation*}
  where $s>\udima(X)$ and the constant $C_2 \ge 1$ depends only on $X$.

  Since each ball in $\WW_k\bigl( X \setminus E; B_0 \bigr)$ is contained in some $B(w_j,10r)$, we conclude that
  \begin{equation*}
    \#\WW_k(X \setminus E; B_0) \le \sum_{j=1}^n \#\WW_k\bigl( X \setminus E; B_0 \cap B(w_j,10r) \bigr) \le C_1C_2n\Bigl( \frac{\delta}{48} \Bigr)^{-s}
  \end{equation*}
  as desired.
\end{proof}

We remark that a different version of the above lemma can be found in \cite[Lemma 3.8]{BjornBjornShanmugalingam2007}. Notice, in particular, that in \cite{BjornBjornShanmugalingam2007} it was assumed that Whitney balls with radii multiplied by $\tfrac12$ are pairwise disjoint, and that there the Whitney balls cover, in general, only a part of the complement of the closed set $E$.

We obtain the following two lemmas as rather immediate consequences of Lemma \ref{lemma:general bound}. The first one (Lemma \ref{lemma:up bound from dim}) gives upper bounds for Whitney ball count in terms of upper dimensions, while in the second one (Lemma \ref{lemma:low bound for dim}) lower bounds for Whitney ball count lead to lower estimates for lower dimensions.

\begin{lemma}\label{lemma:up bound from dim}
   (1) If $E \subset X$ is a closed set and $\udima(E) < \lambda$, then there is $C \ge 1$ such that if $B_0$ is a closed ball of radius $0<R<\diam(E)$ centered at $E$, we have
  \begin{equation*}
    \#\WW_k(X \setminus E; B_0) \le C2^{\lambda k}R^\lambda
  \end{equation*}
  for all $k > -\log_2 R$.

  (2) If $E \subset X$ is a compact set and $\limsup_{r \downarrow 0} \MM^\lambda_r(E) < \infty$ (or $\udimm(E) < \lambda$), then there are $k_0 \in \Z$ and $C \ge 1$ such that
  \begin{equation*}
    \#\WW_k(X \setminus E) \le C2^{\lambda k}
  \end{equation*}
  for all $k \ge k_0$.
\end{lemma}

\begin{proof}
  If $k > -\log_2 R$, the definition of the upper Assouad dimension implies that the set $E \cap B_0$ can be covered by balls $\{ B(w_j,2^{-k}) \}_{j=1}^n$, $w_j \in E$, where $n \le C(2^{-k}/R)^{-\lambda}$
  for some constant $C \ge 1$. The first claim follows by recalling Lemma \ref{lemma:general bound}.

  In the second claim, the assumption on the Minkowski content shows that there are $M>0$ and $r_0>0$ so that $\MM^\lambda_r(E) \le M$ for all $0<r<r_0$. Let $k_0 \in \N$ be such that $2^{-k_0} \le r_0$ and $k \ge k_0$. Since now $E$ has a cover $\{ B(w_j,2^{-k}) \}_{j=1}^n$, $w_j \in E$, for which $n 2^{-\lambda k} \le 2M$, also the second claim follows from Lemma \ref{lemma:general bound}.
\end{proof}

\begin{lemma}\label{lemma:low bound for dim}
  Let $\ell \in \N$, $\lambda \ge 0$, and $c>0$.

  (1) If $E \subset X$ is a closed set and for each closed ball $B_0$ of radius $0<R<\diam(E)$ centered at $E$
  \begin{equation*}
    \#\WW_k(X \setminus E; B_0) \ge c2^{\lambda k}R^\lambda
  \end{equation*}
  for all $k \ge -\log_2 R + \ell$, then $\ldima(E) \ge \lambda$.

  (2) If $E \subset X$ is a compact set and there is $k_0 \in \Z$ for which
  \begin{equation*}
    \#\WW_k(X \setminus E) \ge c2^{\lambda k}
  \end{equation*}
  for all $k \ge k_0$, then $\liminf_{r \downarrow 0} \MM^\lambda_r(E) > 0$ (and thus $\ldimm(E) \ge \lambda$).
\end{lemma}

\begin{proof}
  Let $B_0$ be such a ball, fix $0<r<R2^{-\ell}$, and take $k \in \Z$ such that $2^{-k} < r \le 2^{-k+1}$. Then $k>-\log_2 R + \ell$. Let $\{ B(w_j,r) \}_{j=1}^n$, $w_j \in E$, be a cover of $E \cap B_0$. Then the assumption together with Lemma \ref{lemma:general bound} gives
  \begin{equation*}
    \Bigl( \frac{r}{R} \Bigr)^{-\lambda} \le {2^{\lambda k}}{R^\lambda} \le c^{-1} \#\WW_k(X \setminus E; B_0) \le c^{-1}Cn,
  \end{equation*}
  where $C \ge 1$ is as in Lemma \ref{lemma:general bound}.
  On the other hand, for $R2^{-\ell}\le r<R$, a sufficiently small $c>0$ gives $n \geq 1 \geq cC^{-1}(r/R)^{-\lambda}$
  for all $r$-covers of $E\cap B_0$, and thus we conclude
  $\ldima(E) \ge \lambda$.

  In the second claim, considering a cover of $E$, the same calculation yields $\MM^\lambda_r(E) \ge c_1$ (with $c_1>0$ depending on $X$, $E$, and $c$) for all $0<r<2^{-k_0}$, and the claim follows.
\end{proof}

In order to obtain results in directions converse to the previous lemmas, it is necessary to add some extra conditions on both the space $X$ and the closed set $E\sub X$. We begin again with a general lemma giving a lower bound for the local Whitney ball count in the complement of a porous set.

\begin{lemma} \label{lemma:general low bound}
  If $X$ is $q$-quasiconvex and $E \subset X$ is a closed $\roo$-porous set, then there is a constant $c>0$ depending only on $\roo$, $q$, and the doubling constant $N$ satisfying the following: If $B_0$ is a closed ball of radius $0<R<\diam(E)$ centered at $E$, $0<r<R/2q$, and $\{ B(w_j,r/2) \}_{j=1}^n$, $w_j \in E$, is a maximal packing of $E \cap \frac 1 2 B_0$, then
  \begin{equation*}
    \#\WW_{k}(X \setminus E; B_0) \ge cn,
  \end{equation*}
  where $k \in \Z$ is such that $\roo r/10 < 2^{-k} \le \roo r/5$.
\end{lemma}

\begin{proof}
  Fix $j \in \{ 1,\ldots,N \}$. The porosity assumption implies that there is $y_j \in B(w_j,r)$ satisfying $\dist(y_j,E) \ge \roo r$. By the $q$-quasiconvexity, there is a curve $\gamma_j \colon [0,1] \to X$ such that $\gamma_j(0) = y_j$, $\gamma_j(1) = w_j$, and $\length(\gamma_j) \le qd(y_j,w_j)$. Thus $\gamma_j([0,1]) \subset B(w_j,qr) \subset B_0$, where the latter inclusion holds since $qr\leq R/2$ and $w_j\in \frac 1 2 B_0$. From the continuity of both $\gamma_j$ and the distance function $x \mapsto \dist(x,E)$ it follows that there exists $x_j \in \gamma_j([0,1])$ so that $\dist(x_j,E) = 5\cdot 2^{-k}\leq\roo r$.

  Let $B(z_j,r_j) \in \WW(X \setminus E; B_0)$ be such that $x_j \in B(z_j,r_j)$. Then
  \begin{equation*}
   2^{-k-1} < \tfrac{5}{9} 2^{-k} \le \frac{\dist(x_j,E)}{9} \le r_j \le \frac{\dist(x_j,E)}{7} \le \tfrac 5 7 2^{-k} \leq 2^{-k} < \tfrac {\roo r}2
  \end{equation*}
  and, consequently, $B(z_j,r_j) \in \WW_{k}(X \setminus E; B_0)$ and $B(z_j,r_j) \subset B(w_j,qr+\roo r)$. Since the balls in the collection $\{ B(w_j,r/2) \}_{j=1}^n$ are pairwise disjoint, the doubling condition of $X$ implies that we may decompose $\{ B(w_j,qr+\roo r) \}_{j=1}^n$ into $M \ge 1$ pairwise disjoint collections, where $M$ depends only on $q$, $\roo$, and the doubling constant $N$. Since each ball $B(w_j,qr+\roo r)$ contains a ball from $\WW_{k}(X \setminus E; B_0)$, we conclude that
    $n \le M \#\WW_{k}(X \setminus E; B_0)$
  as desired.
\end{proof}

Lemma \ref{lemma:general low bound} now leads to
lower bounds for Whitney ball count in terms of lower dimensions (Lemma \ref{lemma:low bound from dim}), and upper bounds for upper dimensions follow from an upper bound for Whitney ball count (Lemma \ref{lemma:up bound for dim}).

\begin{lemma}\label{lemma:low bound from dim}
  Assume that $X$ is $q$-quasiconvex.

  (1) If $E \subset X$ is a closed $\roo$-porous set with $\ldima(E) > \lambda$, then there exist $c>0$ and $\ell\in \N$ such that
  \begin{equation*}
    \#\WW_k(X \setminus E; B_0) \ge c2^{\lambda k}R^\lambda
  \end{equation*}
  for all $k > -\log_2 R +\ell$ whenever $B_0$ is a closed ball of radius $0<R<\diam(E)$ centered at $E$.

  (2) If $E \subset X$ is a compact $\roo$-porous set and $\liminf_{r \downarrow 0} \MM^\lambda_r(E)>0$ (or $\ldimm(E) > \lambda$), then there exist $k_0 \in \Z$ and $c>0$ such that
  \begin{equation*}
     \#\WW_k(X \setminus E) \ge c2^{\lambda k}
  \end{equation*}
  for all $k \ge k_0$.
\end{lemma}

\begin{proof}
  Fix such a ball $B_0$. Choose $\ell > \log_2\frac {10q} \roo$  and let $k > -\log_2 R +\ell$ whence $2^{-k} < \frac{\roo R}{10q}$. Define $r=\frac{5}{\roo}2^{-k}<\frac{R}{2q}$ and let $\{ B(w_j,r/2) \}_{j=1}^n$ be a maximal packing of $E\cap \frac 1 2 B_0$. Then $\{ B(w_j,r) \}_{j=1}^n$ is a cover of $E \cap \frac 1 2 B_0$, and hence, $\lambda < \ldima(E)$ implies $n \ge c_0(r/R)^{-\lambda}$. Moreover, as $2^{-k} = \frac{\roo}{5}r$, Lemma \ref{lemma:general low bound} guarantees the existence of a constant $c>0$, independent of $k$, so that
  \begin{equation*}
    \Bigl( \frac{5}{\roo} \Bigr)^{-\lambda}2^{\lambda k}R^\lambda = \Bigl( \frac{r}{R} \Bigr)^{-\lambda} \le c_0^{-1}n \le c_0^{-1}c^{-1} \#\WW_k(X \setminus E; B_0)
  \end{equation*}
  as desired.

  By considering a maximal packing of $E$, a similar calculation shows the second claim.
\end{proof}

\begin{remark}\label{rmk:small c}
In the case when one is dealing with a more general Whitney cover with parameter $\frac 1 3 \leq \delta \leq \frac 1 2$, it might be necessary to consider in Lemma \ref{lemma:general low bound} (and thus also in Lemma \ref{lemma:low bound from dim}) the sum over two consecutive generations $\WW_k$ and $\WW_{k+1}$, since the Whitney cover may lack a certain generation of balls. In particular, this is the case with the usual Euclidean Whitney cube decomposition.
\end{remark}

\begin{remark}
Porosity appears to be rather crucial in the proof of Lemma \ref{lemma:general low bound} (and thus also in Lemma \ref{lemma:low bound from dim}), and indeed, our Example \ref{ex:thick cantor} shows that the lower bound can fail for non-porous sets. The quasiconvexity assumption can be weakened (in particular the existence of rectifiable curves is not needed), but some kind of a `quantitative local connectivity' property is required.

To see this, consider, for instance, the set \[X=\big([0,1]\times A_m\big) \cup \big(\{0\}\times [0,1]\big)\sub \R^2\] equipped with the relative metric, where, for $m\in\N$, $A_m=\{0\}\cup\{2^{-mj}:j\in\N\}$, and $E= [0,1]\times \{0\}\sub X$. Then $E$ is $2^{-2m}$-porous and connected, and $\dimm(E)=1$. Nevertheless, for $k\in\N$ such that $2^{-k+3}\notin A_m$ (i.e., $k-3$ not divisible by $m$),
all Whitney balls in $\WW_k(X\setminus E)$ are centered at $\{0\}\times [0,1]$, and thus $\#\WW_k(X\setminus E)\leq C$ for these $k$. This shows that the estimate of Lemma \ref{lemma:low bound from dim}(2) can not be true for $E$.

However, even without the quasiconvexity assumption we obtain, under the other assumptions of Lemma \ref{lemma:general low bound}, that there exists $\ell\in\N$ (depending on $\roo$) such that
\[\sum_{j=k}^{k+\ell}\#\WW_{j}(X \setminus E; B_0) \ge cn.\]
It follows that all the consequences of that lemma also hold in non-quasiconvex spaces with the corresponding modifications.
\end{remark}

\begin{lemma}\label{lemma:up bound for dim}
  Assume that $X$ is $q$-quasiconvex, and let $\lambda \ge 0$ and $C \ge 1$.

  (1) If $E \subset X$ is a closed $\roo$-porous set such that for each closed ball $B_0$ of radius $0<R<\diam(E)$ centered at $E$ and for every $k \ge -\log_2 R$ we have
  \begin{equation*}
    \#\WW_k(X \setminus E; B_0) \le C2^{\lambda k}R^\lambda,
  \end{equation*}
  then $\udima(E) \le \lambda$.

  (2) If $E \subset X$ is a compact $\roo$-porous set and there exists $k_0 \in \Z$ such that for every $k \ge k_0$ we have
  \begin{equation*}
    \#\WW_k(X \setminus E) \le C2^{\lambda k},
  \end{equation*}
  then $\limsup_{r \downarrow 0} \MM^\lambda_r(E) < \infty$ (and thus $\udimm(E) \le \lambda$).
\end{lemma}

\begin{proof}
  Let $0<R<\diam(E)/2$. Define $r_0 = R/q$, fix $0<r<r_0$, and let $k \in \Z$ be as in Lemma~\ref{lemma:general low bound}. Furthermore, let $\{ B(w_j,r/2) \}_{j=1}^n$, $w_j \in E$, be a maximal packing of $E \cap B_0$. Then, using Lemma~\ref{lemma:general low bound} and the assumption for the ball $2B_0$, we obtain
  \begin{equation*}
    cn \le \#\WW_{k}(X \setminus E; 2B_0) \le C2^{\lambda k}(2R)^\lambda \le C\Bigl( \frac{10}{\roo} \Bigr)^{\lambda}\Bigl( \frac{r}{R} \Bigr)^{-\lambda},
  \end{equation*}
  where $c>0$ is as in Lemma \ref{lemma:general low bound}. Thus $E \cap B_0$ can be covered by at most $C_1({r}/{R})^{-\lambda}$ balls of radius $r$ for all
  $0<r<r_0$, where $C_1 \ge 1$ does not depend on $r$. For $r_0\leq r\leq R\leq \diam(E)$ the same bound follows by simply choosing the constant $C_1$ to be sufficiently large.

  In the second claim, considering a maximal packing of $E \cap B_0$, essentially the same calculation yields $\MM^\lambda_r(E) \le C$ for all $0<r<r_0$.
\end{proof}

If we assume, instead of porosity as in Lemma \ref{lemma:up bound for dim},
only that $E\sub X$ has zero $\mu$-measure, and also drop the quasiconvexity assumption, we obtain (in Lemma \ref{lemma:mu0 bounds} below) weaker estimates, involving the regularity dimensions of $\mu$, for the upper Assouad and Minkowski dimensions of $E$. However, as pointed out in Corollary \ref{coro:regular mu0 bounds}, in $s$-regular spaces the conclusions of Lemma \ref{lemma:mu0 bounds} coincide with those of Lemma \ref{lemma:up bound for dim} under these weaker assumptions. The obvious example of the unit ball $E=B(0,1)\sub\R^d$ shows that the condition $\mu(E)=0$ can not be removed from these considerations.

\begin{lemma}\label{lemma:mu0 bounds}
  Assume that $\mu$ is a doubling measure on $X$, and let $0 \le \lambda < \ldimreg(\mu)$ and $C > 0$.

  (1) If $E \subset X$ is a closed set with $\mu(E)=0$, and for each closed ball $B_0$ of radius $0<R<\diam(E)$ centered at $E$ and for every $k \ge -\log_2 R$ we have
  \begin{equation*}
    \#\WW_k(X \setminus E; B_0) \le C2^{\lambda k}R^\lambda,
  \end{equation*}
  then $\lcodima(E) \ge \ldimreg(\mu)-\lambda$ and $\udima(E) \le \lambda + \udimreg(\mu) - \ldimreg(\mu)$.

  (2) If $E \subset X$ is a compact set with $\mu(E)=0$, and there is $k_0\in\Z$ such that for every $k \ge k_0$ we have
  \begin{equation*}
    \#\WW_k(X \setminus E) \le C2^{\lambda k},
  \end{equation*}
  then $\lcodimm(E) \ge \ldimreg(\mu)-\lambda$ and $\udimm(E) \le \lambda + \udimreg(\mu) - \ldimreg(\mu)$.
\end{lemma}

\begin{proof}
  Fix a ball $B_0$ of radius $R>0$ and let $0<r<R$. Take $k_1 \in \Z$ such that $2^{-k_1} \le r < 2^{-k_1+1}$. Observe that for $\lambda < t < \ldimreg(\mu)$ and $B \in \WW_k(X \setminus E; B_0)$ with $k \ge k_1$ we have $\mu(B) \le C_1\mu(B_0)2^{-tk}R^{-t}$ for some constant $C_1 \ge 1$ not depending on $R$ nor $k$. Since $E_r\cap B_0 \sub E\cup \bigcup_{k = k_1}^\infty \WW_k(X \setminus E; B_0)$, we obtain
  \begin{align*}
   \mu(E_r\cap B_0)
    &\le \mu(E) + C_1\sum_{k=k_1}^\infty \#\WW_k(X \setminus E; B_0) 2^{-tk} R^{-t} \mu(B_0) \\
    &\le CC_1\sum_{k=k_1}^\infty 2^{-k(t-\lambda)} R^{\lambda-t} \mu(B_0)
    \le \frac{CC_1}{1-2^{-(t-\lambda)}}R^{\lambda-t} 2^{-k_1(t-\lambda)} \mu(B_0) \\
    &\le \frac{CC_1}{1-2^{-(t-\lambda)}}\left(\frac r R\right)^{t-\lambda}\mu(B_0).
  \end{align*}
  The first claim follows now from the definition of the lower Assouad codimension and Lemma \ref{lemma:dima ja codima}.

  The above calculation also shows that under the assumptions of the second claim, we have $\mu(E_r)\leq C r^{t-\lambda}$ for all $0<r<2^{-k_0}$ (where the constant $C > 0$ depends on $E$). Hence $\MM^{\mu,t-\lambda}_r(E) \le C$ for all $0<r<2^{-k_0}$, and so $\lcodimm(E) \ge t-\lambda$. The upper bound for $\udimm(E)$ follows from the estimates in
  \eqref{eq:dimm ja codimm}.
\end{proof}

\begin{corollary}\label{coro:regular mu0 bounds}
  Assume that $\mu$ is an $s$-regular measure on $X$.

  (1) If $E \subset X$ is a closed set with $\mu(E)=0$ such that for each closed ball $B_0$ of radius $0<R<\diam(E)$ centered at $E$ and for every $k \ge -\log_2 R$ we have
  \begin{equation*}
    \#\WW_k(X \setminus E; B_0) \le C2^{\lambda k}R^\lambda,
  \end{equation*}
  then $\udima(E) \le \lambda$.

  (2) If $E \subset X$ is a compact set with $\mu(E)=0$ and there exists $k_0 \in \Z$ so that for every $k \ge k_0$ we have
  \begin{equation*}
    \#\WW_k(X \setminus E) \le C2^{\lambda k},
  \end{equation*}
  then $\udimm(E) \le \lambda$.
\end{corollary}

\section{Tubular neighborhoods and spherical dimension}\label{sect:tube}

It is not hard to show that in an $s$-regular space $X$ the results of Theorems \ref{thm:porous char of dimm} and \ref{thm:mu=0 char of dimm} can be reformulated equivalently in terms of the measures of the `annular neighborhoods' $E_{2^{-k}}\setminus E_{2^{-k-1}}$ of a compact set $E\sub X$; recall that
$E_r=\{x\in X : \dist(x,E)< r\}$ is the open (\emph{tubular}) $r$-neighborhood of $E$, sometimes also called the \emph{parallel set} of $E$. Now one may ask what happens when the `width' of the annular neighborhoods is allowed to shrink towards zero, i.e., how $\mu(E_r\setminus E_s)$ behaves when $0<s\uparrow r$. In general metric spaces this does not necessarily lead to meaningful results, but in the Euclidean space $\R^d$ one is then prompted to obtain estimates for the $(d{-}1)$-dimensional `surface area' of the sets $E_r$, that is, for $\Ha^{d-1}(\bdry E_r)$. The set $\bdry E_r=\{x\in\R^d : \dist(x,E) = r\}$ is called the \emph{$r$-boundary} of $E$. We remark that the study of the behavior of $\Ha^{d-1}(\bdry E_r)$ is related to the study of curvature measures of $E$; see e.g.\ \cite{WinterZahle2012}.

In this section, we will show that there is an intimate connection between $\Ha^{d-1}(\bdry E_r)$ and the number of Whitney balls in $\WW_k(\R^d\setminus E)$, where $2^{-k}\approx r$. More precisely, we have the following estimate for each compact set $E\sub\R^d$:
\begin{equation}\label{eq:area and balls}
  c r^{d-1}\#\WW_k(\R^d\setminus E) \le \HH^{d-1}(\partial E_r) \le C r^{d-1}\sum_{j=k+2}^{k+4} \#\WW_j(\R^d \setminus E),
\end{equation}
where $2^{-(k+1)} < r \le 2^{-k}$, and the constants $c,C\geq 0$ depend only on the dimension $d$. These upper and lower estimates for $\Ha^{d-1}(\bdry E_r)$ are established in Propositions \ref{prop:local minkowski tube} and \ref{prop:low bound for d-1}, respectively.

The consequences of \eqref{eq:area and balls} are numerous. For instance, Oleksiv and Pesin \cite{OleksivPesin1985} gave the following general estimate for $\Ha^{d-1}(\bdry E_r)$, when $E\sub\R^d$ is bounded and $r>0$:
\begin{equation}\label{eq:OP}
 \Ha^{d-1}(\bdry E_r)\leq\begin{cases}
                          C_1 r^{d-1}, & \text{ for } r>\diam(E), \\
        C_2 r^{-1}, & \text{ for } 0<r\leq\diam(E),
                         \end{cases}
\end{equation}
where $C_1 \ge 1$ depends only on $d$ and $C_2 \ge 1$ depends only on $d$ and the diameter of $E$. These estimates are now easy consequences of the upper bound in \eqref{eq:area and balls}. Examples from \cite{OleksivPesin1985} show that the growth orders in \eqref{eq:OP} are essentially sharp for a general bounded set $E\sub\R^d$. Nevertheless, as soon as more information is available on the geometry of $E$, our \eqref{eq:area and balls} implies much better estimates on $\Ha^{d-1}(\bdry E_r)$.

Another application of \eqref{eq:area and balls} is related to the so-called \emph{lower spherical dimension} of a compact set $E\sub\R^d$. This was defined by Rataj and Winter \cite{RatajWinter2010} as
\begin{equation*}
  \ldims(E) = \inf\{ \lambda\geq 0 : \liminf_{r\downarrow 0} \frac{\Ha^{d-1}(\bdry E_r)}{r^{d-1-\lambda}} = 0 \}.
\end{equation*}
The corresponding upper dimension, defined as above but with $\limsup$ instead of $\liminf$, was shown in \cite{RatajWinter2010} to agree with the upper Minkowski dimension of $E$ if $\Ha^d(E)=0$ (this result can also be deduced from our more general estimate \eqref{eq:area and balls} as well), but for the lower dimensions Rataj and Winter~\cite{RatajWinter2010} obtained the estimate
\begin{equation}\label{eq:lower sdim}
\tfrac{d-1}{d}\ldimm(E)\leq \ldims(E)\leq \ldimm(E).
\end{equation}
Moreover, in \cite{Winter2011} the bounds in \eqref{eq:lower sdim}, apart from the end point in the lower bound,
were shown to be sharp. We are now able to bring further clarification into these results: First, we have the following consequence of \eqref{eq:area and balls} (compare this to Theorem \ref{thm:mu=0 char of dimm}):

\begin{theorem}\label{thm:char for ldims}
If $E\sub\R^d$ is a compact set, then
  \begin{equation*}
     \ldims(E) = \liminf_{k\to\infty}\tfrac 1 k {\log_2 \#\WW_k(\R^d \setminus E)}.
  \end{equation*}
\end{theorem}

It is also true that the corresponding characterization for the upper spherical dimension holds for all compact $E \subset \R^d$.
By Theorem \ref{thm:porous char of dimm}, Theorem \ref{thm:char for ldims} implies the following equivalence of lower dimensions for porous sets of $\R^d$:

\begin{corollary}\label{coro:s-dim for poro}
If $E \subset \R^d$ is compact and porous,
then $\ldims(E) = \ldimm(E)$.
\end{corollary}

One particular consequence of this result is that if $E \subset \R^d$ is an $s$-regular compact set with $0 \le s < d$, then both the upper and lower spherical dimensions of $E$ are equal to $s$. See also Corollary \ref{coro:saannollisen pullistus} for a more precise statement concerning $s$-regular sets.

Our second result on the lower spherical dimension is that we fix the gap concerning the sharpness of the lower bound in \eqref{eq:lower sdim} in Example \ref{ex:thick cantor}, where we show the following:

\begin{proposition}\label{prop:end-point}
  For each $d \in \N$ there exists a compact set $E\sub \R^d$ with $\HH^d(E)=0$, $\ldimm(E)=d$, and $\ldims(E)=d-1$.
\end{proposition}

This example, featuring a `thick' Cantor type construction, also provides an answer to another question of Winter \cite[Remark 2.4]{Winter2011}, namely that the existence of the Minkowski dimension does not guarantee the equivalence of the lower Minkowski and the lower spherical dimensions. It is also trivial to see that the converse does not hold either: by Corollary \ref{coro:s-dim for poro} it suffices to construct a compact and porous set $E$ with $\ldimm(E) < \udimm(E)$; see \cite[\S 5.3]{Mattila1995}.

Before going into the details of our results, let us briefly mention some of the previous results concerning the $r$-boundaries $\partial E_r$ of a compact set $E\sub\R^d$. These sets are always $(d{-}1)$-rectifiable (see \cite[Lemma 15.13]{Mattila1995} and \cite[Proposition 2.3]{RatajWinter2010}), and have finite $(d{-}1)$-dimensional Hausdorff measure (with estimates), as already seen in \eqref{eq:OP}. If $d \in \{ 2,3 \}$ and $E\sub \R^d$ is compact, then the set $\partial E_r$ is a $(d{-}1)$-Lipschitz manifold for $\HH^{1}$-almost every $r \in (0,\infty)$. This is a result of Brown~\cite{Brown1972} for $d=2$ and Ferry \cite{Ferry1975} for $d=3$. Ferry also gave an example which shows that if $d \ge 4$, then the above claim fails: there exists a Cantor type compact set $E\sub \R^d$ such that $\bdry E_r$ is never a $(d{-}1)$-manifold when $0<r<1$. See also the articles \cite{Fu1985, GariepyPepe1972, RatajZajicek2012} for related results.

It is also true that for all but countably many $r>0$ it holds that
\begin{equation}\label{eq:stacho}
\tfrac{d}{dr}\Ha^{d}(E_r) = C{\Ha^{d-1}(\bdry E_r)};
\end{equation}
see \cite[Corollary 2.5]{RatajWinter2010} and also Stacho \cite{Stacho1976} and the references therein for earlier results related to \eqref{eq:stacho}. The equation \eqref{eq:stacho} is a crucial ingredient in the results of Rataj and Winter \cite{RatajWinter2010} on spherical dimensions. Contrary to this, our approach is completely geometric. In fact, our proofs are based on the following two geometric lemmas, in which we estimate the size of $\partial E_r$ in (or near) a Whitney ball (with a radius comparable to $r$).

\begin{lemma}\label{lemma:tube in cube}
  If $E \subset \R^d$ is a closed set, $k \in \Z$, and $B \in \WW_k(\R^d \setminus E)$, then \begin{equation*}
    \HH^{d-1}(\partial E_r \cap B) \le C2^{-k(d-1)}
  \end{equation*}
  for all $r>0$, where $C \ge 1$ depends only on $d$.
\end{lemma}

\begin{proof}
  Let $B = B(z,R)\in \WW_k(\R^d \setminus E)$.
  We may clearly assume $\#(\bdry E_r\cap B)\geq 2$. Fix two different points $x_1,x_2\in \bdry E_r\cap B$. Let $w_1,w_2\in E$ be such that $|x_i-w_i|=r$, and let $y_i$ be the intersection of $\bdry B$ and the line segment $[x_i,w_i]$ for $i \in \{ 1,2 \}$. Also let $\ell$ be the line through $x_1,x_2$ and let $\proj_\ell$ denote the orthogonal projection onto $\ell$. Then
  \begin{equation}\label{eq:kaakana}
    |x_1-\proj_\ell(w_2)| \ge |x_1-x_2|/2\quad\text{ and }\quad |x_2-\proj_\ell(w_1)| \ge |x_1-x_2|/2,
  \end{equation}
  as otherwise we would have $|x_1-w_2|<r$ or $|x_2-w_1|<r$, which is not possible.

 Our goal is to show that
\begin{equation}\label{eq:lip1}
|y_1-y_2| \ge |x_1 - x_2|/2.
\end{equation}
If
  \begin{equation*}\label{eq:likella*}
    |x_i-\proj_\ell(w_i)| \ge |x_1-x_2|/2
  \end{equation*}
for $i \in \{ 1,2 \}$, then $\proj_\ell(w_i)$, and thus also $\proj_\ell(y_i)$, can not be on the half-line starting from $x_i$ and containing the interval $[x_1,x_2]$.
Hence, by \eqref{eq:kaakana}, it is obvious that
\begin{equation*}\label{eq:lip1*}
|y_1-y_2| \ge |\proj_\ell(y_1) - \proj_\ell(y_2)| \ge |x_1 - x_2|/2,
\end{equation*}
as desired.

We may thus assume that
  \begin{equation}\label{eq:likella}
    |x_i-\proj_\ell(w_i)| \le |x_1-x_2|/2
  \end{equation}
  for $i\in\{1,2\}$.
  Since $|y_i-w_i|\geq \dist(B,E) = 7 R$ and $|x_i-y_i|\leq 2 R$, it follows that $|x_i-y_i|\leq \frac27|y_i-w_i| \le \frac27|x_i-w_i|$. Thus $|x_i - \proj_\ell(y_i)| \le \tfrac27|x_i - \proj_\ell(w_i)|\le \tfrac17|x_1-x_2|$ by \eqref{eq:likella}, and consequently
  \begin{equation} \label{eq:lip2}
  \begin{split}
    |y_1-y_2| &\ge |\proj_\ell(y_1) - \proj_\ell(y_2)| \ge |x_1 - x_2| - |x_1 - \proj_\ell(y_1)| - |x_2 - \proj_\ell(y_2)| \\
    &\ge |x_1-x_2| - \tfrac27|x_1-x_2| = \tfrac57|x_1-x_2| > |x_1 - x_2|/2,
  \end{split}
  \end{equation}
  proving \eqref{eq:lip1}.
  Estimate \eqref{eq:lip1} shows that the above procedure of choosing $y \in \partial B$ for a given $x \in \partial E_r \cap B$ is an inverse of a $2$-Lipschitz mapping from a subset of $\partial B$ onto $\partial E_r \cap B$. Therefore
  \[
    \Ha^{d-1}(\bdry E_r\cap B)\leq 2 \Ha^{d-1}(\bdry B)\leq C2^{-k(n-1)}
  \]
  as claimed.
\end{proof}

We remark that in the case $d=2$, Lemma \ref{lemma:tube in cube} was proved in \cite[Corollary~1]{Jarvi1994} using a different idea.

Our second lemma provides an estimate in the converse direction:

\begin{lemma}\label{lemma:low bound for d-1 in W-balls}
  If $E \subset \R^d$ is a closed set, $k \in \Z$, and $B\in\WW_k(\R^d\setminus E)$, then
  \begin{equation*}
    \HH^{d-1}(\partial E_r \cap 8B) \ge cr^{d-1}
  \end{equation*}
  for all $2^{-k-1} < r\leq 2^{-k}$, where $c>0$ depends only on $d$.
\end{lemma}

\begin{proof}
  Fix $B=B(y,R)\in\WW_k(\R^d\setminus E)$ and $2^{-k-1} < r\leq 2^{-k}$, and
  let $w\in E$ be such that $|y-w|=8R$. Then for all $x\in \partial B(y,4R)$ we have that $\dist(x,E) \geq 4R > 2^{-k}\geq r$, and for all $x\in \partial B(w,2^{-k-1})$ that $\dist(x,E)\leq 2^{-k-1}<r$. Moreover, if $\ell$ is a half-line starting from $y$ and making an angle $0\leq \alpha\leq \alpha_0=\arcsin (1/16)$ with the line segment $[y,w]$ from $y$ to $w$, then $\ell$ intersects both $\partial B(y,4R)$ and $\partial B(w,2^{-k-1})$, and so there exists $x\in\ell\cap \partial E_r\cap B(y,8R)$. Let us consider the radial projection onto $\partial B(y,4R)$, that is, the mapping $\proj \colon \R^d\setminus\{y\} \to \partial B(y,4R)$ for which
  \[
    \proj(x) = y+4R\frac{x-y}{|x-y|}.
  \]
  It is evident that $\proj$ is $1$-Lipschitz in $\R^n\setminus B(y,4R)$ (and thus especially in $\partial E_r \cap B(y,8R)$) and, furthermore, the image $\proj(\partial E_r \cap B(y,8R))$ contains all $x\in \partial B(y,4R)$ for which the angle between the line segments $[y,x]$ and $[y,w]$ is less than $\alpha_0$. It follows that
  \[
    \HH^{d-1}(\partial E_r \cap 8B) \ge \HH^{d-1}(\proj(\partial E_r\cap B(y,8R))) \ge c R^{d-1} \ge cr^{d-1},
  \]
  where $c$ only depends on the dimension $d$.
\end{proof}

We can now prove local quantitative estimates for the $(d{-}1)$-measures of $r$-boundaries;
notice that the global estimates of \eqref{eq:area and balls} are special cases of these local results, applied to sufficiently large balls.

\begin{proposition}\label{prop:local minkowski tube}
  Let $E \subset \R^d$ be a closed set, and let $B_0$ be a closed ball
  centered at $E$.
  If $k \in \Z$ and $2^{-k-1} < r \le 2^{-k}$, then
  \begin{equation*}
    \HH^{d-1}(\partial E_r\cap B_0) \le C r^{d-1}\sum_{j=k+2}^{k+4} \#\WW_j(\R^d \setminus E;B_0),
  \end{equation*}
  where $C \ge 1$ depends only on $d$.
\end{proposition}

\begin{proof}
Fix $r>0$ and let $k\in\Z$ be as above.
If $B \in \WW_j(\R^d \setminus E;B_0)$ is so that $\bdry E_r\cap B \neq \emptyset$, then $2^{-j-1} < r_0 \le r/7 < 2^{-k-2}$ and $2^{-k-5} < r/9 \le r_0 \le 2^{-j}$ where $r_0>0$ is the radius of $B$. Thus
\[
  \bdry E_r\cap B_0\sub\bigcup_{j=k+2}^{k+4} \WW_j(\R^d \setminus E;B_0)
\]
and, consequently, by Lemma \ref{lemma:tube in cube},
\[
  \Ha^{d-1}(\bdry E_r\cap B_0)\leq \sum_{j=k+2}^{k+4} \sum_{B \in \WW_j(\R^d \setminus E;B_0)} \Ha^{d-1}(\bdry E_r\cap B)
  \leq C\sum_{j=k+2}^{k+4} \#\WW_j(\R^d \setminus E;B_0) 2^{-j(d-1)}.
\]
This proves the claim.
\end{proof}

\begin{proposition}\label{prop:low bound for d-1}
  Let $E \subset \R^d$ be a closed set, and let $B_0$ be a closed ball
  centered at $E$.
  If $k \in \N$, and $2^{-k-1}<r\leq 2^{-k}$, then
  \[
    \Ha^{d-1}(\bdry E_r\cap 3B_0)\geq c r^{d-1}\#\WW_k(\R^d\setminus E;B_0),
  \]
  where $c>0$ depends only on $d$.
\end{proposition}

\begin{proof}
By the properties of the Whitney covers and a simple volume argument,
the overlap of the balls $8B$, for $B\in\WW_k(\R^d\setminus E; B_0)$, is uniformly bounded
by a constant $C_1\geq 1$. Moreover, we have for these balls that $8B\sub 3B_0$.
Thus Lemma \ref{lemma:low bound for d-1 in W-balls} yields that
\[
 \Ha^{d-1}(\bdry E_r\cap 3B_0) \geq C_1^{-1} \sum_{B\in\WW_k(\R^d\setminus E;B_0)}\HH^{d-1}(\partial E_r \cap 8B)\geq C_1^{-1}cr^{d-1}\#\WW_k(\R^d\setminus E;B_0),
\]
as desired.
\end{proof}

We record the following consequences of Propositions \ref{prop:local minkowski tube} and \ref{prop:low bound for d-1} for the global behavior of $\partial E_r$.

\begin{proposition}\label{prop:minkowski tube}
  (1) If $E \subset \R^d$ is compact and $\lambda\geq 0$,
  then
  \begin{equation*}
    \HH^{d-1}(\partial E_r)
     \le C r^{d-1-\lambda} \Mi_r^\lambda(E)
  \end{equation*}
  for all $r>0$, where $C \ge 1$ depends only on $d$.

  (2) If $E \subset \R^d$ is compact and $\roo$-porous, and $\lambda\geq 0$, then there is $c>0$ depending only on $d$, $\lambda$, and $\roo$ so that
  \[
    \Ha^{d-1}(\bdry E_r)\geq c r^{d-1-\lambda}\Mi^{\lambda}_{10r/\roo}(E)
  \]
  for all $0<r<\roo\diam(E)/5$.
\end{proposition}

\begin{proof}
(1) Fix $r>0$, let $k \in \Z$ be such that $2^{-(k+1)} < r \le 2^{-k}$, and take $B_0=B(x_0,\diam(E)+r)$ with any fixed $x_0\in E$. Since for $k+2\leq j \leq k+4$ we have $2^{-4}r\leq 2^{-(k+4)}\leq 2^{-j}\leq 2^{-(k+2)}<r$,
Lemma~\ref{lemma:general bound} implies that $\#\WW_j(\R^d \setminus E;B_0)\leq C r^{-\lambda} \Mi_r^\lambda(E)$. The claim now follows from Proposition~\ref{prop:local minkowski tube}, since $\partial E_r=\partial E_r\cap B_0$.

(2) Fix $0<r<\roo\diam(E)/5$, let $k\in\Z$ be such that $2^{-k-1}\leq r<2^{-k}$, and take $B_0=B(x_0,\diam(E))$ with any fixed $x_0\in E$. Let $\{B_j\}_{j=1}^n$ be a maximal $(5r/\roo)$-packing of $E$.
Then, by Lemma~\ref{lemma:general low bound},
$\#\WW_k(\R^d\setminus E;B_0)\geq cn\geq c(10r/\roo)^{-\lambda}\Mi^{\lambda}_{10r/\roo}(E)$,
and thus, by Proposition~\ref{prop:low bound for d-1},
\[
\Ha^{d-1}(\bdry E_r)=\Ha^{d-1}(\bdry E_r\cap 3B_0)
 \geq c_1 r^{d-1-\lambda}\Mi^{\lambda}_{10r/\roo}(E),
\]
where $c_1>0$ depends only on $d$, $\lambda$, and $\roo$.
\end{proof}

Notice that the estimates in \eqref{eq:OP} are easy consequences of Proposition~\ref{prop:minkowski tube}(1). Indeed, if $E\sub\R^d$ is compact, then for $r>\diam(E)$ we have $\Mi^0_r(E)\leq C$, and for $0<r\leq \diam(E)$ that $\Mi^d_r(E)\leq C\diam(E)^d$.

If $0<s<d$ and $E\sub \R^d$ is an $s$-regular compact set, then we can
combine the two cases of Proposition \ref{prop:minkowski tube} into the following corollary; recall that such a set $E$ is necessarily porous, and that $c\leq \Mi^s_r(E)\leq C$ for all $0<r<\diam(E)$. 

\begin{corollary}\label{coro:saannollisen pullistus}
If $E \subset \R^d$ is compact and $s$-regular for $0<s<d$,
then there are $C \ge 1$, $c>0$, and $r_0>0$ so that $cr^{d-1-s} \le \Ha^{d-1}(\bdry E_r) \le Cr^{d-1-s}$ for all $0<r<r_0$.
\end{corollary}

Corollary~\ref{coro:saannollisen pullistus} generalizes the claims concerning the spherical dimension in \cite[Theorem~4.5]{RatajWinter2010}. We thank Steffen Winter for pointing out, after the completion of the first version of this paper, that the claim in Corollary~\ref{coro:saannollisen pullistus} can also be deduced from \cite[Theorem~2.2]{RatajWinter2012}. For a concrete example, see \cite[Example~3.3]{RatajWinter2010}.

Propositions~\ref{prop:local minkowski tube} and \ref{prop:low bound for d-1} can also be used together with the results from Section \ref{sect:whitney} to show the correspondence between Assouad dimensions and local estimates for the $r$-boundaries.

\begin{corollary}\label{coro:assouad tube}
Let $E \subset \R^d$ be a closed set.

(1)  If  $\udima(E) < \lambda$,
then there exists $C\geq 1$ so that
  \[
    \Ha^{d-1}(\bdry E_r\cap B_0)\leq C r^{d-1}\Bigl( \frac{r}{R} \Bigr)^{-\lambda}
  \]
  for all closed balls $B_0$ of radius $0<R<\diam(E)$ centered at $E$ and for every $0<r<R$.

(2)  If there exists $c>0$ and $0<\delta\leq 1$ so that
  \[
    \Ha^{d-1}(\bdry E_r\cap B_0) \geq c r^{d-1}\Bigl( \frac{r}{R} \Bigr)^{-\lambda}
  \]
  for all closed balls $B_0$ of radius $0<R<\diam(E)$ centered at $E$ and for every $0<r<\delta R$, then $\ldima(E)\geq \lambda$.

(3)  If $\HH^{d}(E)=0$ and there exists $C \ge 1$ so that
  \begin{equation*}
    \HH^{d-1}(\partial E_r\cap B_0) \le Cr^{d-1}\Bigl( \frac{r}{R} \Bigr)^{-\lambda}
  \end{equation*}
  for all closed balls $B_0$ of radius $0<R<\diam(E)$ centered at $E$ and for every $0<r<R$, then $\udima(E) \le \lambda$.

(4) If $E$ is porous and $\ldima(E) > \lambda$, then there exists $c>0$ and $0<\delta\leq 1$ so that
  \[
   \Ha^{d-1}(\bdry E_r\cap B_0)\geq cr^{d-1}\Bigl( \frac{r}{R} \Bigr)^{-\lambda}
  \]
  for all closed balls $B_0$ of radius $0<R<\diam(E)$ centered at $E$ and for every $0<r<\delta R$.
\end{corollary}

\begin{proof}
The claim (1) follows directly from Lemma~\ref{lemma:up bound from dim}(1) and Proposition~\ref{prop:local minkowski tube}.
For (2) we need Proposition~\ref{prop:local minkowski tube} together with a slight modification of Lemma~\ref{lemma:low bound for dim}(1), where the assumption allows a sum over a fixed number of consecutive generations of Whitney balls, but the conclusion stays the same. The claim (3) follows from Corollary~\ref{coro:regular mu0 bounds}(1) and Proposition~\ref{prop:low bound for d-1}, and (4) from Lemma~\ref{lemma:low bound from dim}(1) and Proposition~\ref{prop:low bound for d-1}.
\end{proof}

\begin{remark}\label{rmk:converse}
Global versions of the first three cases of Corollary \ref{coro:assouad tube} for compact sets, where Minkowski dimensions are used instead of Assouad dimensions and references to $B_0$ and $R$ are omitted, can be found to be implicit in \cite{RatajWinter2010}; the corresponding version of (1) follows from \cite[Lemma~3.5]{RatajWinter2010}, and versions of (2) and (3) from \cite[Corollary~3.2]{RatajWinter2010}.

There is also a global version of (4), even without the porosity assumption but then with a larger exponent. Namely, it follows from \cite[Proposition 3.7]{RatajWinter2010} that $\ldimm(E)>\lambda$ implies
\[
\Ha^{d-1}(\bdry E_r)\geq c r^{d-1-\lambda\frac{d-1}{d}}
\]
for small $r>0$. This estimate is of course related to the lower bound for spherical dimension in~\eqref{eq:lower sdim}.
The proof of \cite[Proposition 3.7]{RatajWinter2010} is based on the use of the isoperimetric inequality.
However, we now know by Proposition \ref{prop:minkowski tube}(2), that the global analog of (4) holds with the exponent
$d-1-\lambda$ for all porous sets.
\end{remark}

We finish the article by giving the example showing Proposition \ref{prop:end-point}:

\begin{ex}\label{ex:thick cantor}
In this example we exhibit a set $E \subset \R^2$ with $\HH^2(E)=0$ and $\dimh(E) = \ldimm(E) = 2$, but $\ldims(E) = 1$. Besides
proving Proposition \ref{prop:end-point}, and thus giving a positive answer to the first question of Winter in \cite[Remark 2.4]{Winter2011}, the construction shows that Lemma \ref{lemma:general low bound} does not necessarily hold without the porosity assumption, and moreover, that the existence of the Minkowski dimension does not guarantee the equivalence of the lower Minkowski and lower spherical dimensions; the last point was also asked in \cite[Remark 2.4]{Winter2011}. We remark that
this example can be easily generalized to all $\R^d$, $d\geq 1$, with dimensions $\ldimm(E) = d$ and $\ldims(E) = d-1$.

If $Q \subset \R^n$ is a cube, then $\ell(Q)$ is its side-length.
We begin the construction by introducing the following $\lambda$-operation:
\begin{itemize}
  \item[($\lambda$)] If $\QQ$ is any collection of rectangles, then we form a new collection by replacing each $Q \in \QQ$ by four rectangles of side-length $\lambda\ell(Q)$ placed in the corners of $Q$.
\end{itemize}
Let $\Lambda = (\lambda_j)_{j=1}^\infty$, where $\lambda_j = \tfrac12$ for all odd $j$ and $\tfrac14 \le \lambda_j = (\tfrac12)^{1+1/j} < \tfrac12$ for all even $j$. Furthermore, let $S = (s_j)_{j=1}^\infty$ be a sequence of real numbers so that $s_j > 1$ for all $j \in \N$ and $\lim_{j \to \infty}s_j = 1$ and let $(n_j)_{j=1}^\infty$ be a sequence of natural numbers so that $n_{j+1}$ is much bigger than $\sum_{i=1}^j n_i$. We will indicate how to make the precise choice for $n_{j+1}$ during the construction.  Set $\QQ_0 = \{ [0,1]^2 \}$ and for each $j \in \N$ construct $\QQ_j$ recursively from $\QQ_{j-1}$ by applying the $\lambda_j$-operation $n_j$ times. Observe that $\bigcup_{Q \in \QQ_j} Q = \bigcup_{Q \in \QQ_{j-1}} Q$, but $\#\QQ_j = 4^{n_j}\#\QQ_{j-1}$ for all odd $j$. Define $E = \bigcap_{j=1}^\infty \bigcup_{Q \in \QQ_j} Q$.

We define a probability measure $\mu$ on $E$ by dividing the mass of $Q \in \QQ_{j-1}$ evenly for all $4^{n_j}$ subcubes of $Q$ contained in $\QQ_j$. Let $Q_j(x)$ be a rectangle in $\QQ_j$ that contains $x$. Since $\mu(Q) = \prod_{i=1}^j 4^{-n_i}$ and $\ell_j = \ell(Q) = \prod_{i=1}^j \lambda_i^{n_i}$ for all $Q \in \QQ_j$ we have
\begin{align*}
  \liminf_{j \to \infty} \frac{\log\mu(Q_j(x))}{\log\ell(Q_j(x))} &= \liminf_{j \to \infty} \frac{-n_j\log 4 - \sum_{i=1}^{j-1}n_i\log 4}{n_j\log\lambda_j + \sum_{i=1}^{j-1}n_i\log\lambda_i} \\
  &\ge \liminf_{j \to \infty} 2^{-1/j} \frac{n_j\log 4}{n_j(1+1/j)\log 2} = 2
\end{align*}
provided that $n_j$ (depending on $\Lambda$ and $n_1,\ldots,n_{j-1}$) is chosen large enough. Thus $\dimh(E) = 2$ by \cite[Lemma 4.1 and Proposition 3.1]{KaenmakiRajalaSuomala2012b} and \cite[Proposition 2.3]{Falconer1997}.

If $j$ is even, then the distance between any two cubes in $\QQ_j$ is at least $D_j = \lambda_j^{-1}\ell_j - 2\ell_j = \ell_j(\lambda_j^{-1}-2) > 0$. Choose $d_j = \min\{ D_j/3, (\#\QQ_j\ell_j)^{-1/(s_j-1)} \} > 0$. Observe that we may now choose $n_{j+1}$ (depending on $\Lambda$, $S$, and $n_1,\ldots,n_j$) large enough so that the ratio
\begin{equation*}
  \frac{\ell_{j+1}}{d_j} = \frac{\prod_{i=1}^{j+1} \lambda_i^{n_i}}{d_j}
\end{equation*}
is as small as we wish. Thus the length of $\partial E_r$ is at most a constant times $\#\QQ_j \ell_j$ for all $cd_j < r < d_j$, where $c>0$ is as small as we like. The desired estimate $\ldims(E)\leq s_j$
follows from this since $\#\QQ_j \ell_j d_j^{s_j-1}$ remains bounded for all even $j$.

Finally, to prove that $\HH^2(E)=0$ it suffices to show that if $j$ is even, then $\sum_{Q \in \QQ_j} \ell(Q)^2$ can be made arbitrary small by choosing $n_{j}$ large enough. But this is obvious since
$\sum_{Q \in \QQ_j} \ell(Q)^2 = \bigl( \prod_{i=1}^{j-1} (4\lambda_i^2)^{n_i} \bigr)(4\lambda_j^2)^{n_j}
\leq (4\lambda_j^2)^{n_j}$, and here $4\lambda_j^2<1$.
\end{ex}


\end{document}